\documentclass[11pt]{article}
\usepackage{amsmath,amsthm,amssymb}

\def\titlerunning#1{\gdef\titrun{#1}}
\makeatletter
\def\author#1{\gdef\autrun{\def\and{\unskip, }#1}\gdef\@author{#1}}
\def\address#1{{\def\and{\\\hspace*{18pt}}\renewcommand{\thefootnote}{}%
\footnote {#1}}% 
\markboth{\autrun}{\titrun}}
\makeatother
\def\email#1{\hspace*{4pt}{\em e-mail}: #1}
\def\MSC#1{{\renewcommand{\thefootnote}{}%
\footnote{\emph{Mathematics Subject Classification (2010):} #1}}}
\def\keywords#1{\par\medskip
\noindent\textbf{Keywords:} #1}

\newtheorem{theorem}{Theorem}[section]
\newtheorem{prop}[theorem]{Proposition}
\newtheorem{cor}[theorem]{Corollary}

\theoremstyle{definition}

\newtheorem{cons}[theorem]{Construction}
\newtheorem{remark}[theorem]{Remark}

\numberwithin{equation}{section}

\def\S{\mathbb S}

\def\cL{\mathcal L}

\def\cA{\mathcal A}
\def\cC{\mathcal C}

\def\cD{\mathcal D}

\def\cH{\mathcal H}

\def\cM{\mathcal M}

\def\cP{\mathcal P}
\def\cX{\mathcal X}
\def\cY{\mathcal Y}
\def\cU{\mathcal U}
\def\cV{\mathcal V}
\def\cT{\mathcal T}

\def\cS{\mathcal S}
\def\cW{\mathcal W}
\def\cQ{\mathcal Q}
\def\cZ{\mathcal Z}
\def\PG{{\rm PG}}

\def\GF{{\rm GF}}

\def\PGL{{\rm PGL}}

\begin{document}

%%%%% To ease editing, add:

\baselineskip=16pt

%%%%%%%%%%%%%%%%
%% In the running head, give an abbreviation of the title. 
\titlerunning{}

\title{On $q$--covering designs}

\author{Francesco Pavese}

\date{}

\maketitle

\address{F. Pavese: Dipartimento di Meccanica, Matematica e Management, Politecnico di Bari, Via Orabona 4, 70125 Bari, Italy; \email{francesco.pavese@poliba.it}}

\bigskip

\MSC{Primary 51E20, 05B40; Secondary 05B25, 51A05.}

%%%%%%%%

\begin{abstract}
A $q$--covering design $\mathbb{C}_q(n,k,r)$, $k \ge r$, is a collection $\cX$ of $(k-1)$--spaces of $\PG(n-1,q)$ such that every $(r-1)$--space of $\PG(n-1,q)$ is contained in at least one element of $\cX$. Let $\cC_q(n,k,r)$ denote the minimum number of $(k-1)$--spaces in a $q$--covering design $\mathbb{C}_q(n,k,r)$. In this paper improved upper bounds on $\cC_q(2n, 3, 2)$, $n \ge 4$, $\cC_q(3n+8, 4, 2)$, $n \ge 0$, and $\cC_q(2n, 4, 3)$, $n \ge 4$, are presented. The results are achieved by constructing the related $q$--covering designs.

\end{abstract}

\section{Introduction}

Let $q$ be any prime power, let $\GF(q)$ be the finite field with $q$ elements and let $\PG(n-1, q)$ be the $(n-1)$--dimensional projective space over $\GF(q)$. We will use the term $k$--space to denote a subspace of $\PG(n-1,q)$ of projective dimension $k$. Let $t \le s$. A {\em blocking set } $\mathbb B$ is a set of $(t-1)$--spaces of $\PG(n-1,q)$ such that every $(s-1)$--space of $\PG(n-1,q)$ contains at least one element of $\mathbb B$. %A blocking set is called {\em minimal} when no proper subset of it is itself a blocking set. 
In the last fifty years the general problem of determining the smallest cardinality of a blocking set $\mathbb B$ has been studied by several authors (see \cite{KM, CSS} and references therein) and in very few cases has been completely solved \cite{BB, B1, BU, EM, KM1}.  

A blocking set $\mathbb B$ can be seen as a $q$--analog of a well known combinatorial design, called {\em Tur\'an design}, see \cite{EV}, \cite{E}. Indeed, a blocking set $\mathbb B$ is also called a {\em $q$--Tur\'an design} $\mathbb{T}_q(n,t,s)$. The dual structure of a $q$--Tur\'an design $\mathbb{T}_q(n,t,s)$ is called {\em $q$--covering design} and it is denoted with $\mathbb{C}_q(n,n-t,n-s)$. In other words, a $q$--covering design $\mathbb{C}_q(n,k,r)$ is a collection $\cX$ of $(k-1)$--spaces of $\PG(n-1,q)$ such that every $(r-1)$--space of $\PG(n-1,q)$ is contained in at least one element of $\cX$. Let $\cC_q(n,k,r)$ denote the minimum number of $(k-1)$--spaces in a $q$--covering design $\mathbb{C}_q(n,k,r)$. Lower and upper bounds on $\cC_q(n,k,r)$ were considered in \cite{EV}, \cite{E}. Lower bounds are obtained by analytical methods, and upper bounds are obtained by explicit constructions of the related $q$--covering designs. A $q$--covering design $\mathbb{C}_q(n,k,r)$ which cover every $(r-1)$--space exactly once is called {\em $q$--Steiner system}.

The concept of $q$--covering design is of interest not only in projective geometry and design theory, but also in coding theory. Indeed, in recent years there has been an increasing interest in $q$--covering designs due to their connections with constant--dimension codes. In particular, a $q$--Steiner system is an optimal constant--dimension code (so far there is only one known example of $q$--Steiner system, i.e., the $2$--covering design $\mathbb{C}_2(13, 3, 2)$ of smallest possible size \cite{BEOVW}). Observe that, as shown in the inspiring article by  Koetter and Kschischang \cite{KK}, constant--dimension codes can be used for error--correction in random linear network coding theory. 

In this paper we discuss bounds on $q$--covering designs. In Section \ref{1}, based on the $q$--covering design $\mathbb{C}_q(6, 3, 2)$ constructed in \cite{CP}, an improved upper bound on $\cC_q(2n, 3, 2)$, $n \ge 4$, is presented. In the last two sections, starting from a lifted MRD--code, improvements on the upper bounds of $\cC_q(3n+8, 4, 2)$, $n \ge 0$, and $\cC_q(2n, 4, 3)$, $n \ge 4$, are obtained. In particular, first $q$--covering designs $\mathbb{C}_q(8, 4, r)$, $r = 2, 3$, of $\PG(7,q)$ are constructed. Then, by induction, $q$--covering designs $\mathbb{C}_q(3n+8, 4, 2)$, $n \ge 0$, and $\mathbb{C}_q(2n, 4, 3)$, $n \ge 4$, are presented.  
  
In the sequel we will use the following notation $\theta_{n,q}:= \genfrac{[}{]}{0pt}{}{n+1}{1}_q=q^n + \ldots + q + 1$.
    
\section{Preliminaries}

A {\em conic} of $\PG(2,q)$ is the set of points of $\PG(2,q)$ satisfying a quadratic equation: $a_{11} X_1^2 + a_{22}X_2^2 + a_{33}X_3^2 + a_{12}X_1X_2 + a_{13}X_1X_3 + a_{23}X_2X_3 = 0$. There exist four kinds of conics in $\PG(2,q)$, three of which are degenerate (splitting into lines, which could be in the plane $\PG(2,q^2)$) and one of which is non--degenerate, see \cite{H}. 

A {\em regulus} is the set of lines intersecting three skew (disjoint) lines and has size $q+1$. The {\em hyperbolic quadric} $\cQ^+(3,q)$, is the set of points of $\PG(3,q)$ which satisfy the equation $X_1 X_2 + X_3 X_4 = 0$. The hyperbolic quadric $\cQ^+(3,q)$ consists of $(q+1)^2$ points and $2(q+1)$ lines that are the union of two reguli. Through a point of $\cQ^+(3,q)$ there pass two lines belonging to different reguli. 

A {\em line--spread} of $\PG(3,q)$ is a set $\cS$ of $q^2+1$ lines of $\PG(3,q)$ with the property that each point of $\PG(3,q)$ is incident with exactly one element of $\cS$. A {\em $1$--parallelism} of $\PG(3,q)$ is a collection $\cP$ of $q^2+q+1$ line--spreads such that each line of $\PG(3,q)$ is contained in exactly one line--spread of $\cP$. In \cite{B} the author proved that there exist $1$--parallelisms in $\PG(3,q)$.

The {\em Klein quadric} $\cQ^+(5,q)$, is the set of points of $\PG(5,q)$ which satisfy the equation $X_1 X_2 + X_3 X_4 + X_5 X_6 = 0$. The Klein quadric contains $(q^2+1)(q^2+q+1)$ points of and two families each consisting of $q^3+q^2+q+1$ planes called {\em Latin planes} and {\em Greek planes}. Two distinct planes in the same family share exactly one point, whereas planes lying in distinct families are either disjoint or meet in a line. A line of $\PG(5,q)$ not contained in $\cQ^+(5,q)$ is either {\em external}, or {\em tangent}, or {\em secant} to $\cQ^+(5,q)$, according as it contains $0, 1$ or $2$ points of $\cQ^+(5,q)$. A hyperplane of $\PG(5,q)$ contains either $q^3+2q^2+q+1$ or $q^3+q^2+q+1$ points of $\cQ^+(5,q)$. In the former case the hyperplane is called {\em tangent}, contains the $2(q+1)$ planes of $\cQ^+(5,q)$ through one of its points, say $R$, and meets $\cQ^+(5,q)$ in a cone having as vertex the point $R$ and as base a hyperbolic quadric $\cQ^+(3,q)$. In the latter case the hyperplane is called {\em secant} and contains no plane of $\cQ^+(5,q)$. The stabilizer of $\cQ^+(5,q)$ in $\PGL(6,q)$, say $G$, contains a subgroup isomorphic to $\PGL(4,q)$. Also, the stabilizer in $G$ of a plane $g$ of $\cQ^+(5,q)$ contains a subgroup $H$ isomorphic to $\PGL(3,q)$ acting in its natural representation on the points and lines of $g$. For more details see \cite[Chapter 1]{H1}. A {\em Singer cyclic subgroup} of $\PGL(3, q)$ is a cyclic group acting regularly on points and lines of a projective plane $\PG(2,q)$.

\subsection{Lifting an MRD--code}

The set ${\cal M}_{n \times m}(q)$, $n \le m$, of $n \times m$ matrices over the finite field $\GF(q)$ forms a metric space with respect to the {\em rank distance} defined by $d_r(A,B)= \mbox{{\em rank}} (A-B)$. The maximum size of a code of minimum distance $\delta$, with $1 \le \delta \le n$, in $\left({\cal M}_{n\times m}(q),d_r\right)$ is $q^{n(m-\delta+1)}$. A code $\cA \subset \cM_{n \times m}(q)$ attaining this bound is said to be a $(n \times m, \delta)_q$ {\em maximum rank distance code} (or {\em MRD--code} in short). A rank distance code $\cA$ is called {\em $\GF(q)$--linear} if $\cal A$ is a subspace of ${\cal M}_{n \times m}(q)$ considered as a vector space over $\GF(q)$. Linear MRD--codes exist for all possible parameters \cite{Delsarte, Gabidulin, Roth, Sheekey}.

We recall the so--called {\em lifting process} for a matrix $A \in \cM_{n \times m}(q)$, see \cite{SKK}. Let $I_n$ be the $n \times n$ identity matrix. The rows of the $n \times n+m$ matrix $(I_n | A)$ can be viewed as coordinates of points in general position of an $(n-1)$--space of $\PG(n+m-1, q)$. This subspace is denoted by $L(A)$. Hence the matrix $A$ can be ``lifted'' to the $(n-1)$--space $L(A)$. Let $U_i$ be the point of $\PG(n+m-1, q)$ having $1$ in $i$--th position and $0$ elsewhere, $1 \le i \le n+m$, and let $\Sigma$ be the $(m-1)$--space of $\PG(n+m-1,q)$ containing $U_{n+1}, \ldots, U_{n+m}$. Note that if $A \in \cA$, then $L(A)$ is disjoint from $\Sigma$.

\begin{prop} \label{lifting}
\begin{itemize}
\item[i)] If $\cA$ is a $(3 \times m, 2)_q$ MRD--code, $m \ge 3$, then $\cX = \{L(A) \; | \; A \in \cA\}$ is a set of $q^{2m}$ planes of $\PG(m+2, q)$ such that every line of $\PG(m+2, q)$ disjoint from $\Sigma$ is contained in exactly one element of $\cX$. 
\item[ii)] If $\cA$ is a $(4 \times m, 3)_q$ MRD--code, $m \ge 4$, then $\cX = \{L(A) \; | \; A \in \cA\}$ is a set of $q^{2m}$ solids of $\PG(m+3,q)$ such that every line of $\PG(m+3, q)$ disjoint from $\Sigma$ is contained in exactly one element of $\cX$.  
\item[iii)] If $\cA$ is a $(4 \times m, 2)_q$ MRD--code, $m \ge 4$, then $\cX = \{L(A) \; | \; A \in \cA\}$ is a set of $q^{3m}$ solids of $\PG(m+3,q)$ such that every plane of $\PG(m+3, q)$ disjoint from $\Sigma$ is contained in exactly one element of $\cX$.  
\end{itemize}
\end{prop}
\begin{proof}
{\em i)} \ Let $L(A_1), L(A_2) \in \cX$, where $A_1, A_2 \in \cA$ and $A _1 \ne A_2$. Assume by contradiction that $L(A_1) \cap L(A_2)$ contains a line. Then we would have 
$$
\mbox{{\em rank}}
\left(
\begin{array}{c|c}
I_3 & A_1 \\
I_3 & A_2 
\end{array}
\right) \le 4 .
$$
On the other hand,
$$
\mbox{{\em rank}}
\left(
\begin{array}{c|c}
I_3 & A_1 \\
I_3 & A_2 
\end{array}
\right) =  
\mbox{{\em rank}}
\left(
\begin{array}{c|c}
I_3 & A_1 \\
0 & A_2 - A_1 
\end{array}
\right) = 
3 + \mbox{{\em rank}}
\left( A_2 - A_1 \right)  \ge 5 , 
$$
a contradiction. Therefore the set $\cX$ contains $q^{2m}$ planes pairwise meeting in at most a point. Hence there are $(q^2+q+1)q^{2m}$ lines of $\PG(m+2, q)$ lying on a plane of $\cX$. But 
$$
\frac{\left( \theta_{m+2, q} - \theta_{m-1, q} \right) \left( \theta_{m+1, q} - \theta_{m-1, q} \right)}{q+1} = q^{2m}(q^2+q+1)
$$ 
is the total number of lines of $\PG(n+m+1, q)$ that are disjoint from $\Sigma$.

A similar argument ban be used in the remaining cases {\em i)}, {\em ii)}.
\end{proof}

\section{On $\cC_q(2n,3,2)$}\label{1}

In this section we provide an upper bound on $\cC_q(2n, 3,2)$, $n \ge 4$. In \cite{CP}, a constructive upper bound on $\cC_q(6,3,2)$ has been given. In what follows we recall the construction and some of the properties of this $q$--covering design. 

\begin{cons}
Let $g$ be a Greek plane of $\cQ^+(5,q)$. From \cite[Lemma 2.2]{CP}, there exists a set $\cX$ of $q^6-q^3$ planes disjoint from $g$ and meeting $\cQ^+(5,q)$ in a non--degenerate conic that, together with the set $\cY$ of $q^3+q^2+q$ Greek planes of $\cQ^+(5,q)$ distinct from $g$, cover every line $\ell$ of $\PG(5,q)$ that is either disjoint from $g$ or contained in $\cQ^+(5,q) \setminus g$.

Let $\ell$ be a line of $g$. Through the line $\ell$ there pass $q-1$ planes meeting $\cQ^+(5,q)$ exactly in $\ell$ and a unique Latin plane $\pi$. Varying the line $\ell$ over the plane $g$ and considering the planes meeting $\cQ^+(5,q)$ exactly in $\ell$, we get a family $\cZ$ of consisting of $(q-1)(q^2+q+1) = q^3-1$ planes. From \cite[Lemma 2.3]{CP}, every line that is tangent to $\cQ^+(5,q)$ at a point of $g$ is contained in exactly a plane of $\cZ$.

Let $P$ be a point of $\ell$. Through the point $P$ there pass $q$ lines of $\pi$ and $q$ lines of $g$ distinct from $\ell$ and contained in $\cQ^+(5,q)$. Let $S$ be the set of $q^2$ planes generated by a line of $\pi$ through $P$ distinct from $\ell$ and a line of $g$ through $P$ distinct from $\ell$. Let $C$ be a Singer cyclic group of the group $H \simeq \PGL(3,q)$. Here $H$ is a subgroup of $G$ stabilizing the plane $g$. Let $\cT$ be the orbit of the set $S$ under $C$. Then $\cT$ consists of $q^2(q^2+q+1)$ planes and each of these planes has $2q+1$ points in common with $\cQ^+(5,q)$ on two intersecting lines of $\cQ^+(5,q)$. From \cite[Lemma 2.4]{CP}, every line that is secant to $\cQ^+(5,q)$ and has a point on $g$ is contained in exactly one plane of $\cT$. 
\end{cons}

\begin{theorem}\cite[Theorem 2.5]{CP}\label{th1}
The set $\cX \cup \cY \cup \cZ \cup \cT$ is a $q$--covering design $\mathbb{C}_q(6,3,2)$ of size $q^6+q^4+2q^3+2q^2+q-1$.
\end{theorem}

We will need the following result.

\begin{theorem}\label{th2}
There exists a hyperplane $\Gamma$ of $\PG(5,q)$ such that $q^3+2q^2+q-1$ elements of $\cX \cup \cY \cup \cZ \cup \cT$ are contained in $\Gamma$.
\end{theorem}
\begin{proof}
Let $\Gamma$ be a hyperplane of $\PG(5,q)$ containing $g$. Then $\Gamma$ is a tangent hyperplane and contains the planes of $\cQ^+(5,q)$ through a points $R$ of $g$. In particular, there are $q$ planes of $\cY$ contained in $\Gamma$. First of all observe that no plane of $\cX$ is contained in $\Gamma$. Indeed, by way of contradiction, assume that a plane of $\cX$ is contained in $\Gamma$. Then such a plane would meet $g$ in at least a point, contradicting the fact that every plane of $\cX$ is disjoint from $g$. A plane of $\cZ$ that is contained in $\Gamma$ has to contain the point $R$. On the other hand, the $q-1$ planes of $\cZ$, passing through a line of $g$ which is incident with $R$, are contained in $\Gamma$. Hence there are $(q+1)(q-1) = q^2-1$ planes of $\cZ$ contained in $\Gamma$. If $\pi$ is a Latin plane contained in $\Gamma$, then $\pi \cap g$ is a line, say $\ell$. By construction there is a point $P \in \ell$ such that the set $\cT$ contains $q^2$ planes meeting $\pi$ in a line through $P$ and $g$ in a line through $P$. Note that these $q^2$ planes of $\cT$ are contained in $\Gamma$. It follows that there are $q^2(q+1)$ planes of $\cT$ contained in $\Gamma$. The result follows.
\end{proof}

The $q$--covering design of Theorem \ref{th1} can be used recursively to obtain a $q$--covering design $\mathbb{C}_q(2n, 3, 2)$, $n \ge 4$, as described in the following result.

\begin{theorem}\label{th3}
There exists a $q$--covering design $\mathbb{C}_q(2n, 3, 2)$, $n \ge 3$, and a hyperplane $\Gamma$ of $\PG(2n-1,q)$ such that there are $q^{2n-3} + \sum_{j = 0}^{n-2} q^{2(n+j-1)}$ planes of $\mathbb{C}_q(2n,3,2)$ not contained in $\Gamma$ and $(q+1) \left(\sum_{i=2}^{n-1} \left(q^{2i-3} + \sum_{j = 0}^{i-2} q^{2(i+j-1)}\right)\right) - 1$ planes of $\mathbb{C}_q(2n,3,2)$ contained in $\Gamma$. 
\end{theorem}
\begin{proof}
By induction on $n$. If $n = 3$, then, from Theorem \ref{th1} and Theorem \ref{th2}, the result holds true. Assume that the result holds true for $n-1$: there are $q^{4n-10} + q^{4n-12} + \ldots + q^{2n-2} + q^{2n-4} + q^{2n-5}$ planes of a $\mathbb{C}_q(2n-2,3,2)$ not contained in a $(2n-4)$--space $\bar{\Lambda}$ of $\PG(2n-3,q)$ and $(q+1) \left(\sum_{i=2}^{n-2} q^{4i-6} + q^{4i-8} + \ldots + q^{2i} + q^{2i-2} + q^{2i-3} \right) - 1$ planes of $\mathbb{C}_q(2n,3,2)$ contained in $\bar{\Lambda}$. 

We claim that the result holds true for $n$. In $\PG(2n-1,q)$, let $\Lambda$ be the $(2n-4)$--space $\langle U_4, \ldots, U_{2n} \rangle$. Let $\cA$ be a $(3 \times (2n-3), 2)_q$ MRD--code and let $\cU = \{L(A) \; | \; A \in \cA\}$ be the set of $q^{4n-6}$ planes of $\PG(2n-1,q)$ obtained by lifting the matrices of $\cA$. Let $\pi$ be a plane disjoint from $\Lambda$. For every point $P \in \pi$, let $\Lambda_P$ be the $(2n-3)$--space $\langle \Lambda, P \rangle$. From the induction hypothesis there is a $q$--covering design $\mathbb{C}_q(2n-2,3,2)$ of $\Lambda_P$, say $\cD_P$, such that there is a subset of $\cD_P$, say $\bar{\cD}_P$, consisting of the $q^{4n-10} + q^{4n-12} + \ldots + q^{2n-2} + q^{2n-4} + q^{2n-5}$ planes of $\cD_P$ not contained in $\Lambda$. Moreover $|\cD_P \setminus \bar{\cD}_P| = (q+1) \left(\sum_{i=2}^{n-2} q^{4i-6} + q^{4i-8} + \ldots + q^{2i} + q^{2i-2} + q^{2i-3} \right) - 1$. 

Let $\cV = \bigcup_{P \in \pi} \bar{\cD}_P$ and let $\cW = \cD_P \setminus \bar{\cD}_P$. First of all observe that $\cU \cup \cV \cup \cW$ is a $q$--covering design $\mathbb{C}_q(2n,3,2)$. Indeed, from Proposition~\ref{lifting}, every line of $\PG(2n-1,q)$ disjoint from $\Lambda$ is contained in exactly one element of $\cU$. On the other hand, if $r$ is a line of $\PG(2n-1,q)$ meeting $\Lambda$ in at least a point, then $r$ is contained in $\Lambda_P$, for some $P \in \pi$, and $r$ is contained in at least a plane of $\cD_P$. Hence $\cU \cup \cV \cup \cW$ is a $q$--covering design $\mathbb{C}_q(2n,3,2)$. 

Let $\ell$ be a line of $\pi$ and let $\Gamma$ be the hyperplane $\langle \Lambda, \ell \rangle$ of $\PG(2n-1,q)$. We will prove that there are $q^{4n-6} + q^{4n-8} + \ldots + q^{2n} + q^{2n-2} + q^{2n-3}$ planes of $\cU \cup \cV \cup \cW$ not contained in $\Gamma$. Since every plane of $\cU$ is disjoint from $\Lambda$, we have that no plane of $\cU$ is contained in $\Gamma$. The planes of $\cV$ not contained in $\Gamma$ are those of the set $\bigcup_{P \in \pi, P \notin \ell} \bar{\cD}_P$. Moreover the planes of $\cW$ are contained in $\Gamma$. Hence there are 
\begin{multline*}
q^{4n-6} + q^2\left(q^{4n-10} + q^{4n-12} + \ldots + q^{2n-2} + q^{2n-4} + q^{2n-5}\right) \\
= q^{4n-6} + q^{4n-8} + \ldots + q^{2n} + q^{2n-2} + q^{2n-3}
\end{multline*}
planes of $\cU \cup \cV \cup \cW$ not contained in $\Gamma$. Finally note that the planes of $\cU \cup \cV \cup \cW$ contained in $\Gamma$ are those of $\bigcup_{P \in \ell} \bar{\cD}_P \cup \cW$. Hence there are 
\begin{multline*}
(q+1)\left(q^{4n-10} + q^{4n-12} + \ldots + q^{2n-2} + q^{2n-4} + q^{2n-5}\right) + \\
+ (q+1) \left(\sum_{i=2}^{n-2} q^{4i-6} + q^{4i-8} + \ldots + q^{2i} + q^{2i-2} + q^{2i-3} \right) - 1 \\
= (q+1) \left(\sum_{i=2}^{n-1} q^{4i-6} + q^{4i-8} + \ldots + q^{2i} + q^{2i-2} + q^{2i-3} \right) - 1
\end{multline*}
planes of $\cU \cup \cV \cup \cW$ contained in $\Gamma$. 
\end{proof}

Theorem \ref{th3}, together with \cite[Theorem 1]{E} gives the following result.

\begin{cor}\label{lines}
If $n \ge 3$, then 
\begin{multline*}
\left\lceil \frac{\theta_{n-1,q^2} \theta_{2n-2,q}}{q^2+q+1} \right\rceil \le \cC_q(2n,3,2) \le q^{2n-2} \theta_{n-2,q^2} + q^{2n-3} -1 + \sum_{i = 2}^{n-1} \left( \theta_{4i-5, q} - \theta_{2i-4, q} + q^{2i-2} \right) .
%q^{4n-6} + q^{4n-8} + \ldots + q^{2n-2} + q^{2n-3} + \\
%+ (q+1) \left(\sum_{i=2}^{n-1} q^{4i-6} + q^{4i-8} + \ldots + q^{2i-2} + q^{2i-3} \right) - 1 .
\end{multline*} 
\end{cor}

\section{On $\cC_q(3n+8,4,2)$}\label{2}

In this section we provide an upper bound on $\cC_q(3n+8,4,2)$, $n \ge 0$. We first deal with the case $n = 0$.

\begin{cons}
Let $\cA$ be a $(4 \times 4, 3)_q$ MRD--code and let $\cX = \{L(A) \; | \; A \in \cA\}$ be the set of $q^8$ solids of $\PG(7,q)$ obtained by lifting the matrices of $\cA$. Let $\Sigma'$ be the solid of $\PG(7,q)$ containing $U_1, U_2, U_3, U_4$. Then $\Sigma'$ is disjoint from $\Sigma$. Let $\cS = \{\ell_i \; | \; 1 \le i \le q^2+1\}$ be a line--spread of $\Sigma$, let $\cS' = \{\ell_i' \; | \; 1 \le i \le q^2+1\}$ be a line--spread of $\Sigma'$ and let $\mu: \ell_i' \in \cS' \longmapsto \ell_i \in \cS$ be a bijection. Let $\Gamma_i$ denote the $5$--space containing $\Sigma$ and $\ell_i'$, $1 \le i \le q^2+1$. If $\gamma$ is a plane of $\Gamma_i$, then there are $q^2+q$ solids of $\Gamma_i$ meeting $\Sigma_i$ exactly in $\gamma$. Let $\cY_i$ be the set of $q(q+1)^2$ solids of $\Gamma_i$ (distinct from $\Sigma$) meeting $\Sigma$ in a plane containing $\mu(\ell_i')$. Let $\cY = \bigcup_{i=1}^{q^2+1} \cY_i$. Then $\cY$ consists of $q(q+1)^2(q^2+1)$ solids.  
\end{cons}

\begin{theorem}\label{lines1}
The set $\cX \cup \cY$ is a $q$--covering design $\mathbb{C}_q(8,4,2)$ of size $q^8+q(q+1)^2(q^2+1)$.
\end{theorem}
\begin{proof}
Let $r$ be a line of $\PG(7,q)$. If $r$ is disjoint from $\Sigma$, then from Proposition \ref{lifting}, we have that $r$ is contained in exactly one element of $\cX$. If $r$ meets $\Sigma$ in one point, say $P$, then let $\Lambda$ be the $4$--space $\langle \Sigma, r \rangle$, let $\ell_j$ be the unique line of $\cS$ containing $P$, let $P'$ be the point $\Sigma' \cap \Lambda$ and let $\ell_k'$ be the unique line of $\cS'$ containing $P'$. If $j = k$, then $P \in \ell_k$ and $r$ is contained in the $q+1$ solids $\langle \alpha, r \rangle$ of $\cY$, where $\alpha$ is a plane of $\Sigma$ containing $\ell_k$. If $j \ne k$, then $P \notin \ell_k$. Let $\beta$ be the plane of $\Sigma$ containing $\ell_k$ and $P$. Then $r$ is contained in $\langle \beta, r \rangle$, where $\langle \beta, r \rangle$ is a solid of $\cY$. Finally let $r$ be a line of $\Sigma$, then $r$ is contained in $q(q+1)^2$ solids of $\cY$. 
\end{proof}

\begin{remark}
Let $\cL$ be a Desarguesian line--spread of $\PG(7,q)$. There are $(q^4+1)(q^4+q^2+1)$ solids of $\PG(7,q)$ containing exactly $q^2+1$ lines of $\cL$. If $\cZ$ denotes the set of these solids, then it is not difficult to see that every line of $\PG(7,q)$ is contained in at least a solid of $\cZ$. In \cite[p. 221]{KM}, K. Metsch posed the following question:
 ``Is $(q^4 +1)(q^4 +q^2 +1)$ the smallest cardinality of a set of $3$--spaces of $\PG(7,q)$ that cover every line?'' Theorem \ref{lines1} provides a negative answer to this question.
\end{remark}

\begin{remark}
When $q = 2$, the result of Theorem \ref{lines1} was obtained in \cite[Theorem 13]{E}.
\end{remark}

\begin{prop}\label{hyp}
There exists a hyperplane $\Gamma$ of $\PG(7,q)$ such that precisely $q(q+1)(2q+1)$ members of $\cX \cup \cY$ are contained in $\Gamma$.
\end{prop}
\begin{proof}
Let $\Gamma$ be a hyperplane of $\PG(7,q)$ containing $\Sigma$. Then no element of $\cX$ is contained in $\Gamma$, otherwise such a solid would meet $\Sigma$, contradicting the fact that every solid in $\cX$ is disjoint from $\Sigma$. The hyperplane $\Gamma$ intersects $\Sigma'$ in a plane $\sigma$. The plane $\sigma$ contains exactly one line of $\cS'$, say $\ell_k'$. Hence the $q(q+1)^2$ solids of $\cY$ meeting $\Sigma$ in a plane through the line $\mu(\ell_k') = \ell_k$ are contained in $\Gamma$. Let $\ell_j' \in \cS'$, with $j \ne k$, then $\ell_j' \cap \sigma$ is a point, say $R$. In this case the $q+1$ solids generated by $R$ and a plane of $\Sigma$ through $\mu(\ell_j') = \ell_j$ is contained in $\Gamma$. Since the elements of $\cY$ are those contained in the $5$--space $\langle \Sigma, \ell_i' \rangle$, where $\ell_i' \in \cS'$, and meeting $\Sigma$ in a plane through $\ell_i$, the proof is complete.  
\end{proof}

The $q$--covering design of Theorem \ref{lines1} can be used recursively to obtain a $q$--covering design $\mathbb{C}_q(3n+8, 4, 2)$, $n \ge 1$, as described in the following result.

\begin{theorem}\label{th4}
There exists a $q$--covering design $\mathbb{C}_q(3n+8, 4, 2)$, $n \ge 0$, and a hyperplane $\Gamma$ of $\PG(3n+7,q)$ such that there are $q^{3n+2} (2q^2-1) + \sum_{j=0}^{n+1} q^{3(n+j)+5}$ solids of $\mathbb{C}_q(3n+8, 4, 2)$ not contained in $\Gamma$ and $(q^2+q+1) \left(\sum_{i=0}^{n-1} \left(q^{3i+2}(2q^2-1) + \sum_{j=0}^{i+1} q^{3(i+j)+5} \right) \right) + q(q+1)(2q+1)$ solids of $\mathbb{C}_q(3n+8, 4, 2)$ contained in $\Gamma$. 
\end{theorem}
\begin{proof}
By induction on $n$. If $n = 0$, then, from Theorem \ref{lines1} and Proposition \ref{hyp}, the result holds true. Assume that the result holds true for $n-1$: there are $q^{6n+2} + q^{6n-1} + \ldots + q^{3n+5} + q^{3n+2} + 2 q^{3n+1}-q^{3n-1}$ solids of a $\mathbb{C}_q(3n+5, 4, 2)$ not contained in a $3(n+1)$--space $\bar{\Lambda}$ of $\PG(3n+4,q)$ and $(q^2+q+1) \left(\sum_{i=0}^{n-2} q^{6i+8} + q^{6i+5} + \ldots + q^{3i+8} + q^{3i+5} + 2 q^{3i+4} - q^{3i+2} \right) + q(q+1)(2q+1)$ solids of $\mathbb{C}_q(3n+5, 4, 2)$ contained in $\bar{\Lambda}$. 

We claim that the result holds true for $n$. In $\PG(3n+7,q)$, let $\Lambda$ be the $3(n+1)$--space $\langle U_5, \ldots, U_{3n+8} \rangle$. Let $\cA$ be a $(4 \times (3n+4), 3)_q$ MRD--code and let $\cU = \{L(A) \; | \; A \in \cA\}$ be the set of $q^{6n+8}$ solids of $\PG(3n+7,q)$ obtained by lifting the matrices of $\cA$. Let $\Pi$ be a solid disjoint from $\Lambda$. For every point $P \in \Pi$, let $\Lambda_P$ be the $(3n+4)$--space $\langle \Lambda, P \rangle$. From the induction hypothesis there is a $q$--covering design $\mathbb{C}_q(3n+5, 4, 2)$ of $\Lambda_P$, say $\cD_P$, such that there is a subset of $\cD_P$, say $\bar{\cD}_P$, consisting of the $q^{6n+2} + q^{6n-1} + \ldots + q^{3n+5} + q^{3n+2} + 2 q^{3n+1}-q^{3n-1}$ solids of $\cD_P$ not contained in $\Lambda$. Moreover $|\cD_P \setminus \bar{\cD}_P| = (q^2+q+1) \left(\sum_{i=0}^{n-2} q^{6i+8} + q^{6i+5} + \ldots + q^{3i+8} + q^{3i+5} + q^{3i+2} - q^{3i+2} +2q^{3i+4} \right) + q(q+1)(2q+1)$.

Let $\cV = \bigcup_{P \in \pi} \bar{\cD}_P$ and let $\cW = \cD_P \setminus \bar{\cD}_P$. First of all observe that $\cU \cup \cV \cup \cW$ is a $q$--covering design $\mathbb{C}_q(3n+8, 4, 2)$. Indeed, from Proposition~\ref{lifting}, every line of $\PG(3n+7,q)$ disjoint from $\Lambda$ is contained in exactly one element of $\cU$. On the other hand, if $r$ is a line of $\PG(3n+7,q)$ meeting $\Lambda$ in at least a point, then $r$ is contained in $\Lambda_P$, for some $P \in \Pi$, and $r$ is contained in at least a solid of $\cD_P$. Hence $\cU \cup \cV \cup \cW$ is a $q$--covering design $\mathbb{C}_q(3n+8, 4, 2)$. 

Let $\sigma$ be a plane of $\Pi$ and let $\Gamma$ be the hyperplane $\langle \Lambda, \sigma \rangle$ of $\PG(3n+7,q)$. Since every solid of $\cU$ is disjoint from $\Lambda$, we have that no solid of $\cU$ is contained in $\Gamma$. The solids of $\cV$ not contained in $\Gamma$ are those of the set $\bigcup_{P \in \Pi, P \notin \sigma} \bar{\cD}_P$. Furthermore the solids of $\cW$ are contained in $\Gamma$. Hence there are 
\begin{multline*}
q^{6n+8} + q^3 \left(q^{6n+2} + q^{6n-1} + \ldots + q^{3n+5} + q^{3n+2} + 2 q^{3n+1} - q^{3n-1}\right) \\
= q^{6n+8} + q^{6n+5} + \ldots + q^{3n+8} + q^{3n+5} + 2 q^{3n+4} - q^{3n+2}
\end{multline*} 
solids of $\cU \cup \cV \cup \cW$ not contained in $\Gamma$. Finally note that the solids of $\cU \cup \cV \cup \cW$ contained in $\Gamma$ are those of $\bigcup_{P \in \sigma} \bar{\cD}_P \cup \cW$. Hence there are 
\begin{multline*}
(q^2+q+1)\left(q^{6n+2} + q^{6n-1} + \ldots + q^{3n+5} + q^{3n+2} + 2 q^{3n+1}-q^{3n-1}\right) + \\
+ (q^2+q+1) \left(\sum_{i=0}^{n-2} q^{6i+8} + q^{6i+5} + \ldots + q^{3i+8} + q^{3i+5} + 2 q^{3i+4} - q^{3i+2} \right) + q(q+1)(2q+1) \\
= (q^2+q+1) \left(\sum_{i=0}^{n-1} q^{6i+8} + q^{6i+5} + \ldots + q^{3i+8} + q^{3i+5} + 2 q^{3i+4} - q^{3i+2} \right) + q(q+1)(2q+1) 
\end{multline*}
solids of $\cU \cup \cV \cup \cW$ contained in $\Gamma$. 
\end{proof}

Taking into account Theorem \ref{th4} and \cite[Theorem 1]{E}, the following result holds true.

\begin{cor}
If $n \ge 0$, then 
\begin{multline*}
\left\lceil \frac{\theta_{3n+7,q} \theta_{3n+6,q}}{(q+1)(q^2+1)(q^2+q+1)} \right\rceil \le \cC_q(3n+8, 4, 2) \le q^{3n+5} \theta_{n+1,q^3} + \sum_{i = 0}^{n-1} \left( \theta_{6i+10, q} - \theta_{3i+4, q} \right) + \\
+ \sum_{i=0}^{n} \left( q^{3i+2} (2q^2-1) \right) + q(q+1)(2q+1) . 
\end{multline*} 
\end{cor}

\section{On $\cC_q(2n, 4, 3)$}\label{3}

The main goal of this section is to give an upper bound on $\cC_q(2n, 4, 3)$, $n \ge 4$. We begin by providing a construction in the case $n = 4$.

\begin{cons}
Let $\cA$ be a $(4 \times 4, 2)_q$ MRD--code and let $\cX = \{L(A) \; | \; A \in \cA\}$ be the set of $q^{12}$ solids of $\PG(7,q)$ obtained by lifting the matrices of $\cA$. Let $\Sigma'$ be the solid of $\PG(7,q)$ containing $U_1, U_2, U_3, U_4$. Then $\Sigma'$ is disjoint from $\Sigma$. Let $\cP = \{\cS_i \; | \; 1 \le i \le q^2+q+1\}$ be a $1$--parallelism of $\Sigma$, let $\cP' = \{\cS_i' \; | \; 1 \le i \le q^2+q+1\}$ be a $1$--parallelism of $\Sigma'$ and let $\mu: \cS_i' \in \cP' \longmapsto \cS_i \in \cP_i$ be a bijection. For a line $\ell'$ of $\Sigma'$, let $\Gamma_{\ell'}$ denote the $5$--space containing $\Sigma$ and $\ell'$. Since $\cP'$ is a $1$--parallelism of $\Sigma'$, there exists a unique $j$, with $1 \le j \le q^2+q+1$, such that $\ell' \in \cS_j'$. Note that $\mu(\cS_j') = \cS_j$ is a line--spread of $\Sigma$. Let $\ell$ be a line of $\cS_j$ and let $\cY_{\ell}$ be the set of $q^4$ solids of $\Gamma_{\ell'}$ (distinct from $\Sigma$) meeting $\Sigma$ exactly in $\ell$. Let $\cZ_{\ell'} = \bigcup_{\ell \in \cS_j} \cY_{\ell}$. Then $\cZ_{\ell'}$ consists of $q^4(q^2+1)$ solids. Varying $\ell'$ among the lines of $\Sigma'$, we get a set 
$$
\cZ = \bigcup_{\ell' \mbox{ {\em line of} } \Sigma'} \cZ_{\ell'}
$$ 
consisting of $q^4(q^2+1)^2(q^2+q+1)$ solids. 
\end{cons}

\begin{theorem}\label{planes}
The set $\cX \cup \cZ \cup \{\Sigma\}$ is a $q$--covering design $\mathbb{C}_q(8,4,3)$ of size $q^{12}+q^4(q^2+1)^2(q^2+q+1)+1$.
\end{theorem}
\begin{proof}
Let $\pi$ be a plane of $\PG(7,q)$. If $\pi$ is disjoint from $\Sigma$, then, from Proposition \ref{lifting}, we have that $\pi$ is contained in exactly one element of $\cX$. If $\pi$ meets $\Sigma$ in one point, say $P$, then let $\Lambda$ be the $5$--space $\langle \Sigma, \pi \rangle$ and let $\ell'$ be the line of $\Sigma'$ obtained by intersecting $\Sigma'$ with $\Lambda$. Note that $\Lambda = \Gamma_{\ell'}$. Let $\cS_j'$ be the unique line--spread of $\cP'$ containing $\ell'$. Then there exists a unique line $\ell$ of $\cS_j = \mu(\cS_j')$ such that $P \in \ell$ and $\pi$ is contained in $\langle \pi, \ell \rangle$, that is a solid of $\cZ$. If $\pi$ meets $\Sigma$ in a line, say $r$, then let $\cS_k$ be the unique line--spread of $\cP$ containing $r$ and let $\Lambda$ be the $4$--space $\langle \Sigma, \pi \rangle$. Then $\Lambda \cap \Sigma'$ is a point, which belongs to a unique line, say $r'$, of the line--spread $\mu^{-1}(\cS_k) = \cS_k'$ of $\cP'$. Since there are $q^2$ solids of $\Gamma_{r'}$ meeting $\Sigma$ exactly in $r$ and containing $\pi$, we have that in this case $\pi$ is contained in $q^2$ members of $\cZ$.
Finally if $\pi$ is a plane of $\Sigma$, then $\pi$ is contained in $\Sigma$. 
\end{proof}

\begin{remark}
Note that Theorem \ref{planes} generalizes the result of \cite[Theorem 16]{E}.
\end{remark}

\begin{prop}\label{hyp1}
There exists a $5$--space $\Lambda$ of $\PG(7,q)$ containing exactly $q^4(q^2+1)+1$ members of $\cX \cup \cZ \cup \{\Sigma\}$. Moreover every hyperplane of $\PG(7,q)$ through $\Lambda$ contains precisely $q^4(q^2+1)(q^2+q+1)+1$ solids of $\cX \cup \cZ \cup \{\Sigma\}$. 
\end{prop}
\begin{proof}
Let $\Lambda$ be a $5$--space containing $\Sigma$. Then $\Lambda$ meets $\Sigma'$ in a line, say $r$, and $\Lambda = \langle \Sigma, r \rangle$. The line $r$ belongs to a unique line--spread $\cS_i'$ of the $1$--parallelism $\cP'$ of $\Sigma'$. Then $\mu(\cS_i') = \cS_i$ is a line--spread belonging to the $1$--parallelism $\cP$ of $\Sigma$. The $q^4(q^2+1)$ solids of $\cZ$ lying in $\langle \Sigma, r \rangle$ meet $\Sigma$ in a line of $\cS_i$ and are contained in $\Lambda$. Let $s$ be a line of $\Sigma'$ such that $s \ne r$. In this case none of the $q^4(q^2+1)$ solids of $\cZ$ lying in $\langle \Sigma, s \rangle$ is contained in $\Lambda$. Indeed, assume by contradiction that there is a solid $\Delta$ contained in $\Lambda \cap \langle \Sigma, s \rangle$, then $\Delta \subset \langle \Sigma, s \cap r \rangle$ and hence $\Delta \cap \Sigma$ is a plane of $\Sigma$, contradicing the fact that every solid of $\cZ$ meets $\Sigma$ in a line. On the other hand, no solid of $\cX$ is contained in $\Lambda$, otherwise such a solid would meet $\Sigma$ not trivially. Finally note that $\Sigma$ is a solid of $\Lambda$.

Let $\Gamma$ be a hyperplane of $\PG(7,q)$ through $\Lambda$. Then $\Gamma \cap \Sigma'$ is a plane, say $\sigma$, containing the line $r$. Repeating the previous argument for every line of the plane $\sigma$, it turns out that there are $q^4(q^2+1)(q^2+q+1)$ solids of $\cZ$ in $\Gamma$, as required.    
\end{proof}

Let $\S_4$ denotes $\cX \cup \cZ \cup \{\Sigma\}$. Similarly to \cite[Theorem 17]{E}, $\S_4$ can be used as a base for a recursive construction of a $q$--covering designs $\mathbb{C}_q(2n, 4, 3)$, $n \ge 5$.   

\begin{theorem}
Let $\S_n$ be a $q$--covering design $\mathbb{C}_q(2n, 4, 3)$, $n \ge 4$, such that there is a $(2n-3)$--space of $\PG(2n-1, q)$, say $\Lambda_n$, containing precisely $\alpha_n$ elements of $\S_n$ and every hyperplane of $\PG(2n-1, q)$ through $\Lambda_n$ contains $\beta_n$ members of $\S_n$. Then there exists a $q$--covering design $\mathbb{C}_q(2n+2, 4, 3)$, say $\S_{n+1}$, where 
$$
|\S_{n+1}| = q^{6(n-1)} + (q^2+1)(q^2+q+1) |\S_n| - q(q+1)^2(q^2+1) \beta_n + q^3(q^2+q+1) \alpha_n.
$$ 
Moreover there exists a $(2n-1)$--space of $\PG(2n+1, q)$, say $\Lambda_{n+1}$, containing $\alpha_{n+1} = |\S_n|$ elements of $\S_{n+1}$ and such that every hyperplane of $\PG(2n+1, q)$ through $\Lambda_{n+1}$ contains $\beta_{n+1}$ members of $\S_{n+1}$, where 
$$
\beta_{n+1} = (q^2+q+1)|\S_n| - (q^3+q^2+q) \beta_n - q^2 \alpha_n. 
$$
\end{theorem}
\begin{proof}
Let $\Lambda_n$ be the $(2n-3)$--space of $\PG(2n+1, q)$ generated by $U_5, \dots, U_{2n+2}$, let $\cA$ be a $(4 \times (2n-2), 2)$ MRD--code and let $\cU$ be the set of $q^{6(n-1)}$ solids obtained by lifting the matrices of $\cA$. Let $\Gamma$ be a solid disjoint from $\Lambda_n$. For a line $\ell$ of $\Gamma$ there exists a $q$--covering design $\mathbb{C}_q(2n, 4, 3)$, say $\S_n$, of $\langle \Lambda_n, \ell \rangle$ such that $\alpha_n$ elements of $\S_n$ are contained in $\Lambda_n$ and every $2(n-1)$--space of $\langle \Lambda_n, \ell \rangle$ through $\Lambda_n$ contains $\beta_n$ members of $\S_n$. Let $\cV_\ell$ be the set of solids of $\S_n$ not contained in a $2(n-1)$--space of $\langle \Lambda_n, \ell \rangle$ through $\Lambda_n$. Then $\cV_\ell$ consists of $|\S_n| - \beta_n - q(\beta_n - \alpha_n)$ solids. Varying $\ell$ among the lines of $\Gamma$, we obtain the following set of solids: 
$$
\cV = \bigcup_{\ell \mbox{ {\em line of} } \Gamma} \cV_{\ell} .
$$ 
For a point $P$ of $\Gamma$, there are $\beta_n$ solids of $\S_n$ contained in $\langle \Lambda_n, P \rangle$, among which $\alpha_n$ are contained in $\Lambda_n$. Let $\cW_P$ denotes the set of $\beta_n - \alpha_n$ solids of $\S_n$ contained in $\langle \Lambda_n, P \rangle$ and not contained in $\Lambda_n$. Varying $P \in \Gamma$ we get the following set of solids: 
$$
\cW = \bigcup_{P \in \Gamma} \cW_{P} .
$$ 
Let ${\bar \cW}$ be set of $\alpha_n$ solids of $\S_n$ contained in $\Lambda_n$ and let $\S_{n+1} = \cU \cup \cV \cup \cW \cup \bar{\cW}$. We claim that $\S_{n+1}$ is a $q$--covering design $\mathbb{C}_q(2n+2, 4, 3)$. Let $\pi$ be a plane of $\PG(2n+1,q)$. If $\pi$ is disjoint from $\Lambda_n$, then, from Proposition \ref{lifting}, there is a unique solid of $\cU$ containing $\pi$. If $\pi$ meets $\Lambda_n$ in a point, then $\langle \Lambda_n, \pi \rangle$ is a $2(n-1)$--space meeting the solid $\Gamma$ in a line, say $r$. Then there is a solid of $\cV_{r}$ containing $\pi$. Finally, if $\pi$ shares with $\Lambda_n$ at least a line, then there is at least a solid of $\cW \cup \bar{\cW}$ containing $\pi$. 

By construction it follows that
\begin{multline*}
|\S_{n+1}| = q^{6(n-1)} + (q^2+1)(q^2+q+1)\left( |\S_n| - \beta_n - q (\beta_n - \alpha_n) \right) + (q+1)(q^2+1) (\beta_n - \alpha_n) + \alpha_n \\
= q^{6(n-1)} + (q^2+1)(q^2+q+1) |\S_n| - q(q+1)^2(q^2+1) \beta_n + q^3(q^2+q+1) \alpha_n.
\end{multline*}
In order to complete the proof, let us denote with $\Lambda_{n+1}$ the $(2n-1)$--space $\langle \Lambda_n, s \rangle$, where $s$ is a fixed line of $\Gamma$. The solids of $\S_{n+1}$ that are contained in $\Lambda_{n+1}$ are those of $\cV_s \cup \bigcup_{P \in s} \cW_P \cup \bar{\cW}$ that are exactly the solids of $\S_n$. Hence $\alpha_{n+1} = |\S_n|$. A hyperplane $\cH$ of $\PG(2n+1,q)$ through $\Lambda_{n+1}$ meets $\Gamma$ in a plane, say $\sigma$, where $s \subset \sigma$. Then the solids of $\S_{n+1}$ contained in $\cH$ are the solids contained in 
$$
\left(\bigcup_{\ell \mbox{ {\em line of} } \sigma} \cV_{\ell} \right) \cup \left( \bigcup_{P \in \sigma} \cW_{P} \right) \cup \bar{\cW}.
$$ 
Therefore 
\begin{multline*}
\beta_{n+1} = (q^2+q+1) \left( |\S_n| - \beta_n - q (\beta_n - \alpha_n) \right) + (q^2+q+1) (\beta_n - \alpha_n) + \alpha_n \\
= (q^2+q+1) |\S_n| - q(q^2+q+1) \beta_n - q^2 \alpha_n.
\end{multline*} 
\end{proof}

\begin{cor}
$$
q^{12} + q^2(q^4+1)(q^2+1)(q^2+q+1) + 1 \le \cC_q(8, 4, 3) \le q^{12} + q^4(q^2+1)^2(q^2+q+1) + 1
$$
\begin{multline*}
(q^4+1)(q^6+q^3+1)(q^8+q^6+q^4+q^2+1) \le \cC_q(10, 4, 3) \le \\
q^{18} + q^4(q^2+1)(q^2+q+1)(q^8+q^6+q^4+q^3+q^2+1) + 1
\end{multline*}
\end{cor}

%\bigskip
%{\footnotesize
%\noindent\textit{Acknowledgments.}
%This work was supported by Ministry for Education, University and Research of Italy MIUR (Project PRIN 2012 ``Geometric %Structures, Combinatorics and their Applications'') and by the Italian National Group for Algebraic and Geometric Structures and %their Applications (GNSAGA-- INdAM).}

\end{document}